\documentclass[12pt, a4paper]{amsart}

\usepackage{array}
\usepackage{amsxtra}
\usepackage[final]{graphicx}
\usepackage{natbib}
\usepackage{verbatim}
\usepackage{amsmath}
\usepackage{amsfonts}
\usepackage{amssymb}
\usepackage{mathtools}
\usepackage[colorlinks=true, bookmarksnumbered=true,pdfpagemode=None,final]
{hyperref}

\usepackage[all]{xy}
\usepackage{tikz-cd}

\usepackage{amssymb}

\usepackage{hyperref}
\usepackage{mathrsfs}

\newtheorem{Theorem}[equation]{Theorem}
\newtheorem{Corollary}[equation]{Corollary}
\newtheorem{Lemma}[equation]{Lemma}
\newtheorem{Proposition}[equation]{Proposition}

\theoremstyle{definition}

\numberwithin{equation}{section}

\newcommand{\C}{{\mathbb C}}

\begin{document}
	
	\title[Generalized flag variety]{Blow-up of a generalized flag variety}
	
	\author[I. Biswas]{Indranil Biswas}
	
	\address{Department of Mathematics, Shiv Nadar University, NH91, Tehsil Dadri, Greater Noida, Uttar Pradesh, 201314, India}
	
	\email{indranil.biswas@snu.edu.in, indranil29@gmail.com}
	
	\author[P. Saha]{Pinakinath Saha}
	
	\address{Department of Mathematics, Indian Institute of Technology, Delhi, Hauz Khas, New Delhi-110016, India.}
	
	\email{pinakinath@iitd.ac.in}

	
	\thanks{Corresponding Author: P.S}
	
	\begin{abstract}
		Let $G$ be a connected simply connected semisimple complex algebraic group and $P\, \subset\, G$
		a parabolic subgroup. We give a necessary and sufficient condition for a line bundle --- on the blow-up
		of the generalized flag variety $G/P$ along a smooth Schubert variety --- to be 
		ample (respectively, nef). Furthermore, it is shown that every such nef line bundle is actually globally generated. As
		a consequence, we are able to describe when such a blow-up is (weak) Fano.
	\end{abstract}
	
	\maketitle
	\tableofcontents	
	
	\section{Introduction}\label{S1}
	
To a complete algebraic variety $X$, one associates the group $N_{1}(X)$ of all one-cycles modulo numerical equivalence; this is a free abelian group of finite 
rank (cf. \cite[Section 1.4.C]{Lazarsfeld}). The corresponding real vector space ${N_{1}(X)}_{\mathbb{R}}\,=\,
N_{1}(X)\otimes_{\mathbb Z} {\mathbb R}$ contains the cone of effective one-cycles, denoted by 
${\rm NE}(X)$, which is the convex cone generated by classes of all closed irreducible curves in $X$. There has been considerable interest in studying various cones 
of positive curves and divisors on algebraic varieties. However, it is generally difficult to describe the entire cone ${\rm NE}(X)$, or its dual, which is the 
nef cone.
	
	Let $G$ be a connected, simply connected, semisimple affine algebraic group over $\mathbb{C}$.
	Take a parabolic subgroup $P\, \subset\, G$. The quotient $X\,:=\,G/P$ is the generalized flag variety associated
	to $(G,\, P)$. Let $Z\, \subset\, X$ be a smooth Schubert variety (cf. Section \ref{S2}) of codimension $c$, with $c\,\ge 2$. Consider the blow-up
	\[{\rm Bl}_Z\ :\ {\rm Bl}_{Z}X\ \longrightarrow\
	X\]
	of $X$ along $Z$. The exceptional divisor for the above map ${\rm Bl}_Z$ is denoted by $E_Z$. Note 
	that ${\rm Bl}_{Z}X$ is also smooth because both $Z$ and $X$ are smooth.
	
	A smooth projective variety $X$ is called {\it Fano} (respectively, {\it weak-Fano}) if its anticanonical bundle $-K_{X}$ is ample (respectively, nef and big).
	
	This paper arose from the following natural question: When is the blow-up ${\rm Bl}_{Z}X$ of $X\,=\,G/P$ along $Z$ 
	Fano or (weak) Fano?
	
	We prove the following:
	
	\begin{enumerate}
		\item The nef cone and the ample cone of ${\rm Bl}_{Z}X$ are computed
		(cf. Theorem~\ref{Theorem:blow-up} and Corollary~\ref{cor: ample}).
		\vspace{.1cm}
		
		\item A sufficient condition is given under which the anticanonical bundle of ${\rm Bl}_{Z}X$ is big
		(cf. Theorem~\ref{Theorem: Big}).
		\vspace{.1cm}
		
		\item We describe when ${\rm Bl}_{Z}X$ is (weak) Fano (cf. Corollary~\ref{cor: Fano}).
		\vspace{.1cm}

		\item More elaborate descriptions are given for the Grassmannian and its generalization, namely
		the cominuscule Grassmannian (cf. Proposition \ref{cor: Grassmannian}, 
		Proposition~\ref{prop: cominuscule}, and Corollary~\ref{cor: cominuscule}).
	\end{enumerate}
	
	In our setting, the nef cone is polyhedral and admits an explicit description. Moreover, the cone of globally 
	generated line bundles coincides with the nef cone --- a phenomenon that does not hold in general.
	
	In a recent work \cite{hu2025mirror}, Hu, Ke, Li and Song investigated the case where $Z$ is a smooth Schubert 
	subvariety of the Grassmannian ${\rm Gr}(r,n)$, and established that the blow-up $ {\rm Bl}_Z {\rm Gr}(r,n)$ is 
	Fano if and only if the codimension of $Z$ is at most $n$. Using root system–theoretic techniques, we recover 
	their result \cite[Theorem 1.1]{hu2025mirror} and further generalize it to the setting of arbitrary generalized 
	flag varieties (cf.~Corollary~\ref{cor: grassFano} and ~Corollary~\ref{cor: grassweakFano}).
	
	\section{Notation and Preliminaries}\label{S2}
	
	This section recalls some notation and preliminaries. For details on matters related to algebraic geometry, algebraic groups, 
	Lie algebras, we refer to \cite{Lazarsfeld}, \cite{Hartshorne}, \cite{Borel}, \cite{Billey-Lakshmibai}, \cite{BrionKumar}, and 
	\cite{Humphreys72}.
	
	Fix a maximal torus $T$ of 
	$G.$ Let $W\,=\,N_{G}(T)/T$ be the
	Weyl group of $G$ with respect to $T.$ Denote by $R$ the set of roots of $G$ with respect to $T$. Let $B^{+}$
	be a Borel subgroup of $G$ containing $T$. The Borel subgroup of $G$ opposite to $B^{+}$, determined by $T$, is
	denoted by $B$; in other words, we have $B\,=\,n_{0}B^{+}n_{0}^{-1}$, where $n_{0}$ 
	is a representative in $N_{G}(T)$ of the longest element $w_{0}$ of $W$. 
	
	Let $R^{+}\,\subset\, R$ be the set
	of all positive roots of $G$ with respect to the Borel subgroup $B^{+}$, and
	$R^{-} \,:=\, -R^{+}$ is the set of negative roots.
	Let $S \,=\, \{\alpha_1,\,\cdots,\,\alpha_{\ell}\}$ denote the set of simple roots in $R^{+},$
	where $\ell$ is the rank of $G.$ The simple reflection in $W$ corresponding to $\alpha_i$ is denoted by $s_{i}.$
	
	Denote by $X^{*}(T)$ (respectively, $X_{*}(T)$) the group of the characters (respectively, cocharacters) of $T.$
	We have a natural perfect bilinear pairing
	\begin{equation}\label{ep}
		\langle \cdot,\, \cdot \rangle\,:\, X^{*}(T)\times X_{*}(T)\,
		\longrightarrow\, \mathbb{Z}.
	\end{equation}
	For a root $\alpha\,\in\, R,$ denote
	the corresponding coroot in ${X_{*}(T)}$ by $\alpha^{\vee}$. Let $X^{*}(T)^+$ denote the set of dominant characters
	of $T$ with respect to $B^{+}$ (i.e., $\lambda\, \in\, X^{*}(T)$ such that $\langle \lambda,\, \alpha_{i}^{\vee}
	\rangle\,\in\, \mathbb{Z}_{\ge 0}$ for all $1\,\le\, i\,\le\, \ell$). For any simple root $\alpha_{i}$,
	denote the fundamental weight corresponding to $\alpha_{i}$ by $\varpi_{i},$ in other words, $\langle \varpi_{i},\, 
	\alpha_{j}^{\vee}\rangle\,=\,\delta_{ij},$ where $\delta_{ij}$ denotes the Kronecker delta. The half sum of all 
	positive roots of $G$ will be denoted by $\rho$.
	
	Fix a parabolic subgroup $P$ of $G$ containing $B$. Define $S_{P}\,:=\,\{\alpha\,\in\, S\,\,\big\vert\,\, 
	U_\alpha\,\subset\, P\}$. Let $W_{P}$ be the Weyl subgroup of $W$ generated by $s_{\alpha}$ for all 
	$\alpha\,\in\, S_{P}$. For each $\beta$ in $R$, there exists a unique connected $T$--stable subgroup 
	$U_{\beta}$ of $G$. Then $P$ is the subgroup of $G$ generated by $B$ and all the root subgroups $U_{\alpha}$,
	$\alpha\,\in\, S_P$.
	
	Let $X\,:=\,G/P$ be the generalized flag variety associated to $(G,\, P)$. The (opposite) Schubert varieties in
	$X\,=\,G/P$ are parameterized by the set $W^P$ of all the minimal left coset 
	representatives of $W/W_P$ in $W$. The Schubert variety corresponding to $w\,\in\, W^{P}$ is 
	\[\overline{BwP}/P\ \subseteq\ G/P.\]
	The opposite Schubert variety corresponding to $w\,\in\, W^{P}$ is 
	\[\overline{B^{+}wP}/P\ \subseteq\ G/P.\]
	
	The quotient map $$p\ :\ G\ \longrightarrow\ X\ =\ G/P$$ makes $G$ a principal $P$--bundle over $X$. Denote by
	$X^{*}(P)$ the group of all characters of $P$. Note that 
	$$X^{*}(P)\ =\ \{\lambda\,\in\, X^{*}(T) \,\,\big\vert\,\, \langle \lambda,\, {\alpha}^{\vee}\rangle\,
	=\, 0 \ \, \forall\ \, \alpha \,\in\, S_P\}.$$
	
	For any $\lambda\,\in\, X^{*}(P)$, there is a natural linear action of $P$ on the one dimensional vector space
	$\mathbb{C}_{\lambda}\,:=\,\mathbb{C}$, namely,
\begin{equation}\label{z1}
p\cdot x\ =\ \lambda(p)x
\end{equation}
for all $p\,\in\, P$ and $x\,\in\, \mathbb{C}$. 
	
	For any $\lambda\,\in\, X^{*}(P)$, set $L_{\lambda}\,:=\, (G\times \mathbb{C}_{\lambda})\big/\sim$ with the
	equivalence relation $\sim$ being defined by $(g,\, x)\,\sim\, (gp^{-1},\, px)$, for $g\,\in\, G$,\,
	$p\,\in\, P$ and $x\,\in\, \mathbb{C}$ (see \eqref{z1}). Then ${L}_{\lambda}$ is the total space of the line bundle, over $X$,
	associated to the principal $P$--bundle $p$ for the character $\mathbb{C}_{\lambda}$ of $P$. Thus we obtain a map
	\begin{equation}\label{eq: picard}
		L\ :\ X^{*}(P)\ \xrightarrow{\,\,\,\sim\,\,\,}\ {\rm Pic}(X)
	\end{equation}
	that sends any $\lambda\, \in\, X^{*}(P)$ to the above line bundle $L_{\lambda}$ constructed using $\lambda$.
	In fact, the homomorphism $L$ in \eqref{eq: picard} is an isomorphism. Consequently, ${\rm Pic}(X)$ is a free
	finitely generated abelian group. Moreover, ${\rm Pic}(X)$ is generated by the globally generated line
	bundles $L_{\varpi_{\alpha}}$, where $\alpha\,\in\, S\setminus S_{P}$.
	
	On the other hand, for $\alpha\,\in\, S\setminus S_{P}$, consider the opposite Schubert divisors
	\[D_{\alpha}\ :=\, \overline{B^{+}s_{\alpha}P}/P\ \subset\ G/P\]
	in $G/P$. We have $L_{\varpi_{\alpha}}\,=\, \mathcal{O}_{X}(D_{\alpha})$. Under the
	isomorphism $L$ in \eqref{eq: picard},
	\[\mathcal{O}_{X}(1)\ =\ \mathcal{O}_{X}(\sum_{\alpha\in S\setminus S_{P}}^{}D_{\alpha}),\] which is, in fact, a
	very ample line bundle on $X$.
	
	For $\alpha\,\in\, S\setminus S_{P}$, consider the Schubert curves $C_{\alpha}\,:=\,\overline{Bs_{\alpha}P/P}$. Then
	the cone of effective one cycle is freely generated by $\{C_{\alpha}\}_{\alpha\in S\setminus S_{P}}$.
	We have the intersection pairing 
	\begin{equation}\label{intersection}
		D_{\alpha} \cdot C_{\beta}\ =\
		\begin{cases}
			1 &\ \text{ if } \alpha \,=\, \beta, \\
			0 &\ \text{ if } \alpha \,\neq\, \beta.
		\end{cases}	
	\end{equation}
	
	\section{Nef cone of the blow-up }\label{S3}
	
	Take a smooth Schubert variety $Z\,\subset\, X$ of codimension $c$, with $c\,\ge\, 2$. Let
	\[{\rm Bl}_Z\ :\ {\rm Bl}_{Z}X\ \longrightarrow\
	X\] be the blow-up map of $X$ along $Z$, with $E_Z\, \subset\, {\rm Bl}_{Z}X$ being the exceptional divisor for the map ${\rm Bl}_Z$.
	
	The canonical line bundle of $X$ will be denoted by $K_X$. For simplicity, we also denote a canonical divisor of $X$ by 
	$K_{X}$.
	
	The following lemma relates the canonical bundle $K_{{\rm Bl}_{Z}X}$ of ${\rm Bl}_{Z}X$ with $K_X$.
	
	\begin{Lemma}[{cf. \cite[Chap. II, Ex. 8.5.]{Hartshorne}}] \label{lemma: canonical}
		Let $X$, ${\rm Bl}_ZX$ and ${\rm Bl}_Z$ be as above. Then
		\[K_{{\rm Bl}_Z X}\,\ =\,\ ({\rm Bl}_Z)^{*} K_{X}\otimes \mathcal{O}_{{\rm Bl}_Z X}((c-1)E_Z),\]
		where $c$ is the codimension of $Z$ in $X$.
	\end{Lemma}
	
	\begin{Corollary}\label{Cor: anti-cano}
		The anticanonical line bundle of ${\rm Bl}_{Z}X$ is given by
		\[K^{-1}_{{\rm Bl}_{Z} X}\,\ =\,\ ({\rm Bl}_{Z})^* K^{-1}_{X}\otimes \mathcal{O}_{{\rm Bl}_{Z}X}(-(c-1)E_Z),\]
		where $c$ is the codimension of $Z$ in $X$.
	\end{Corollary}
	
	Recall that $\mathcal{O}_{X}(1)\,:=\,\mathcal{O}_X(\sum_{\alpha\in S\setminus S_{P}}D_{\alpha})$.
	Let $H\,=\,{\rm Bl}_{Z}^{*}\mathcal{O}_{X}(1)$.
	
	\begin{Theorem}\label{Theorem:blow-up} 
		Let ${\rm Bl}_{Z}X$ be the blow-up of $X$ along $Z$. Then
		the line bundles ${\rm Bl}_{Z}^*D_\alpha,$ with $\alpha\,\in\, S\setminus S_P$,
		and $H - E_{Z}$ are globally generated over ${\rm Bl}_{Z}X$. Moreover, the nef cone
		${\rm Nef}({\rm Bl}_{Z} X)$ is generated by $\{{\rm Bl}_{Z}^*D_\alpha\}_{\alpha\in S\setminus S_P}\cup
		\{H - E_{Z}\}.$ Its dual ${\rm NE}({\rm Bl}_{Z} X),$ namely the Mori cone of curves, is generated
		by the classes of $\{\widetilde{C}_\alpha\}_{\alpha\in S\setminus S_{P}}$ and $e$,
		where $\widetilde{C}_\alpha \,:=\, ({\rm Bl}_{Z})^{*}(C_\alpha)-e$ with $e$ being the class of a line in the
		fiber of the projection $E_Z\, \longrightarrow\, Z$. Moreover, the cone of globally
		generated line bundles on ${\rm Bl}_{Z}X$ coincides with the cone ${\rm Nef}({\rm Bl}_{Z}X)$. 
	\end{Theorem}
	
	\begin{proof}
		The pullback ${\rm Bl}_{Z}^{*}(D_{\alpha})$ is globally generated for every $\alpha\, \in\, S\setminus S_P$,
		because $D_{\alpha}$ is globally generated. Consider the short exact sequence of sheaves on $X$
		\[0\,\longrightarrow\, \mathcal{O}_X(-Z)\,\longrightarrow\, \mathcal{O}_{X}
		\,\longrightarrow\, \mathcal{O}_{X}\big\vert_{Z}\,\longrightarrow\, 0 .\] Tensoring it with $\mathcal{O}_{X}(1)$,
		we have the exact sequence
		\begin{equation}\label{f1}
			0\,\longrightarrow\, \mathcal{O}_{X}(1)\otimes \mathcal{O}_X(-Z)\,\longrightarrow\, \mathcal{O}_{X}(1)
			\,\longrightarrow\, \mathcal{O}_{X}(1)\big\vert_{Z}\,\longrightarrow\, 0 .
		\end{equation}
		{}From the very ampleness of $\mathcal{O}_{X}(1)$ it follows that the restriction map
		\[H^0(X,\, \mathcal{O}_{X}(1))\ \longrightarrow\ H^0(Z,\, \mathcal{O}_{X}(1)\big\vert_{Z})\]
		is surjective \cite[Theorem 3]{Mehta-Ramanathan}. Thus from \eqref{f1} there is the following short exact sequence
		\begin{equation}\label{f2}
			0\, \longrightarrow\, H^0(X,\, \mathcal{O}_{X}(1)\otimes \mathcal{O}_{X}(-Z))
			\, \longrightarrow\, H^0(X,\, \mathcal{O}_{X}(1))\, \longrightarrow\, H^0(Z,\,
			\mathcal{O}_{X}(1)\big\vert_{Z})\, \longrightarrow\, 0.
		\end{equation}
		Since $\mathcal{O}_{X}(1)$ is ample, by \cite[Theorem 6]{Littelmann} it follows that the sheaf 
		$\mathcal{O}_X(1) \otimes \mathcal{O}(-Z)$ is globally generated and $Z$ is cut out by the sections of
		$H^0(X,\, \mathcal{O}_{X}(1)\otimes \mathcal{O}_{X}(-Z))$. 
		Note that ${{\rm Bl}_{Z}}_{*}\mathcal{O}_{{\rm Bl}_{Z}X}\,=\,\mathcal{O}_{X}$, and hence using the projection
		formula, 
		\[ H^0(X,\,\mathcal{O}_{X}(1)\otimes \mathcal{O}_{X}(-Z))\ =\ 
		H^0({\rm Bl}_{Z}X,\, {\rm Bl}^*_{Z}(\mathcal{O}_{X}(1)\otimes \mathcal{O}_{X}(-Z)))
		\]
		\[
		=\ H^0({\rm Bl}_{Z}X,\, \mathcal{O}_{{\rm Bl}_{Z}(X)}(H-E_{Z})).\]
		As $Z$ is cut out by the sections of $H^0(X,\mathcal{O}_{X}(1)\otimes \mathcal{O}_{X}(-Z))$, we know that
		$H-E_{Z}$ is base point free. Hence $H-E_{Z}$ is globally generated; in particular, $H-E_{Z}$ is nef.
		Thus we have
		\[
		\{({\rm Bl}_{Z})^*D_\alpha\,\,\big\vert\,\, \alpha\,\in\, S\setminus S_P\} \cup \{H - E_{Z}\}
		\ \subset\ {\rm Nef}({\rm Bl}_{Z} X).\]
		Note that ${\rm Pic}({\rm Bl}_{Z}(X))$ is a free abelian group generated by $\{{\rm Bl}_{Z}^*D_\alpha\}_{\alpha
			\in S\setminus S_P}$ and $E$. Consequently, for any nef divisor $D$ of ${\rm Bl}_{Z}X$, there exists integers
		$\{a_\alpha\}_{\alpha \in S\setminus S_P}$ and $b$ such that
		\[D\,\ =\,\ \sum_{\alpha \in S\setminus S_P} a_{\alpha}{\rm Bl}_{Z}^*D_\alpha + b E_{Z}.\]
		
		Denote by $e$ the class of a line in the fiber of the projection $E_Z\, \longrightarrow\, Z$.
		We have $D\cdot e\,\ge\, 0$, because $D$ is nef. Since $E_{Z}\cdot e\,=\,-1$, and ${\rm Bl}_{Z}^{*}D_{\alpha}\cdot e
		\,=\,0$ for all $\alpha \,\in \,S\setminus S_{P}$, we conclude that $b$ is non-positive. Set $m\,=\,-b$. 
		
		For $\alpha\,\in\, S\setminus S_{P}$, set \[\widetilde{C_\alpha}\,\ =\,\ {\rm Bl}_{Z}^{*}C_{\alpha}-e.\] Recall that
		using projection formula and \eqref{intersection}, \[{\rm Bl}_{Z}^{*}D_{\alpha} \cdot{\rm Bl}_{Z}^{*}C_{\beta}
		\,\ =\,\ \Bigg\{\begin{matrix}
			1,\,\, ~\alpha\,=\,\beta\\
			0,\,\, ~ \alpha\,\neq\,\beta
		\end{matrix}\] for all $\alpha, \,\beta \,\in\, S\setminus S_{P}$.
		
		Further, because $D$ is nef, we have \[D\cdot \widetilde{C_\alpha}\ \ge\ 0 \] for $\alpha\,\in\, S\setminus S_{P}$. 
		Consequently, $a_{\alpha}\,\ge\, m$ for all $\alpha\,\in\, S\setminus S_{P}$. Set $m_{\alpha}\,=\,a_{\alpha}-m$ for
		all $\alpha \,\in\, S\setminus S_{P}$. Then we have 
		\[D\ =\ \sum_{\alpha \in S\setminus S_P} m_{\alpha}{\rm Bl}_{Z}^*D_\alpha + m(H-E_{Z}).\]
		Therefore, \[{\rm Nef}({\rm Bl}_{Z} X)\ =\ \bigg \langle \{{\rm Bl}_{Z}^*D_\alpha \,\, \big\vert\,\,
		\alpha\,\in\, S\setminus S_P \}\, \cup \{ H - E_{Z}\} \bigg \rangle.\]
		Observe that \[{\rm Bl}_{Z}^{*}D_{\alpha}\cdot \widetilde{C_\beta}\ =\ \Bigg\{\begin{matrix}
			1,\,\, ~\alpha\,=\,\beta\\
			0,\,\, ~ \alpha\,\neq\,\beta
		\end{matrix}\] and \[(H-E_{Z})\cdot e\ =\ 1.\] Therefore, the dual cone
		${\rm NE}({\rm Bl}_{Z} X)$ of ${\rm Nef}({\rm Bl}_{Z}X)$ is generated by the classes of
		$$\{\widetilde{C_\alpha}\}_{\alpha \in S\setminus S_{P}} \ \, \text{and }\ \, e.$$ 
		
		The proof of the last statement follows immediately.
	\end{proof}
	
	\begin{Corollary}\label{cor: ample}
		A line bundle $L$ on ${\rm Bl}_{Z}X$ is ample if and only if there exist positive integers
		$\{a_{\alpha}\}_{\alpha \in S\setminus S_P}$, and $b\,>\,0$, such that
		\[L\ =\ \sum_{\alpha \in S\setminus S_P}a_{\alpha} {\rm Bl}_{Z}^*D_\alpha\, + b (H-E_{Z}).\]
	\end{Corollary}
	
	\begin{proof}
		First assume that $L$ is ample. Then, by Theorem \ref{Theorem:blow-up}, the result follows.
		
		To prove the converse, assume that
		\[L\ =\ \sum_{\alpha \in S\setminus S_P}a_{\alpha} {\rm Bl}_{Z}^*D_\alpha\, + b (H-E_{Z}),\] where $a_{\alpha} \,>
		\, 0$ for all $\alpha \,\in\, S \setminus S_P$ and $b\, >\, 0$. By Theorem \ref{Theorem:blow-up}, the cone of curves
		${\rm NE}({\rm Bl}_{Z} X)$ is generated by the classes of $\{\widetilde{C_\alpha}\}_{\alpha \in S\setminus S_{P}}$
		and $e$. Since $a_{\alpha} \,> \, 0$ for all $\alpha \,\in\, S \setminus S_P$, and $b\, >\, 0$,
		and \[{\rm Bl}_{Z}^{*}D_{\alpha}\cdot \widetilde{C_\beta}\ =\ \Bigg\{\begin{matrix}
			1,\,\, ~\alpha\,=\,\beta\\
			0,\,\, ~ \alpha\,\neq\,\beta
		\end{matrix}\] \[(H-E_{Z})\cdot e\ =\ 1,\] applying Kleiman's criterion for ampleness it follows
		that $L$ is ample.
	\end{proof}
	
	\begin{Lemma}\label{lemma: big}
		The divisor $H-E_{Z}$ on ${\rm Bl}_{Z}X$ is big.
	\end{Lemma}
	
	\begin{proof}
		{}From Theorem \ref{Theorem:blow-up} it follows that $H-E_{Z}$ is an effective Cartier divisor on ${\rm Bl}_ZX$. 
		As ${{\rm Bl}_{Z}}_{*}\mathcal{O}_{{\rm Bl}_{Z}X}\,=\,\mathcal{O}_{X}$, using the projection
		formula, 
		\[ H^0(X,\,\mathcal{O}_{X}(1)\otimes \mathcal{O}_{X}(-Z))\ =\ H^0({\rm Bl}_{Z}X,\,
		\mathcal{O}_{{\rm Bl}_{Z}(X)}(H-E_{Z})).\] Since $H-E_{Z}$ is effective, there is a nonzero section
		$0\,\neq\, s \,\in\, H^0(X,\, \mathcal{O}_{X}(1)\otimes \mathcal{O}_{X}(-Z))$ such that the zero-locus $Z(s)$ of
		$s$ contains $Z$. This implies that the divisor $H-E_{Z}$ is linearly equivalent to the strict transform
		$\widetilde{Z(s)}$ of $Z(s)$ in ${\rm Bl}_{Z}X$. From the ampleness of $\mathcal{O}_{X}(1)$
		it follows that the open subset $X-Z(s)$
		is affine. Moreover ${\rm Bl}_{Z}X- \widetilde{Z(s)}$ is isomorphic to $X-Z(s)$, and hence
		${\rm Bl}_{Z}X- \widetilde{Z(s)}$ is affine. Therefore, by \cite[p.~61, Lemma 2.4]{Brion22} it follows that
		$H-E_Z$ is a big divisor.
	\end{proof}
	
	For $\alpha\,\in\, S\setminus S_P$, define $$\beta_\alpha\ :=\ \langle w_{0,P}(\rho),\, {\alpha}^{\vee} \rangle ,$$
	where $\langle \cdot,\, \cdot \rangle$ is the pairing in \eqref{ep}, and $w_{0,P}$ is the longest element of $W_P$.
	
	\begin{Theorem}\label{Theorem: Big}
		Let $X$,\, ${\rm Bl}_ZX$ and ${\rm Bl}_Z$ be as before, such that $\beta_\alpha + 2-c\,\ge\, 0$ for all
		$\alpha\,\in\, S\setminus S_P$. Then the anticanonical line bundle of ${\rm Bl}_{Z}(X)$ is big.
	\end{Theorem}
	
	\begin{proof}
		The isomorphism $L$ in \eqref{eq: picard} takes
		the anticanonical line bundle $-K_{X}$ of $X$ to $\lambda\,=\,\rho+w_{0,P}(\rho)$
		\cite[Chapter 3, Section 3.1, p.~85]{BrionKumar}.
		Observe that the line bundle ${L}_{\lambda}$ corresponds to the divisor 
		\[\sum_{\alpha\in S\setminus S_{P}} (1+\beta_\alpha)D_{\alpha}.\] 
		Thus by Corollary \ref{Cor: anti-cano}, the anticanonical divisor $-K_{{\rm Bl}_{Z}(X)}$ of ${\rm Bl}_{Z}(X)$ is
		\[
		-K_{{\rm Bl}_{Z}(X)}\ =\ \sum_{\alpha\in S\setminus S_{P}} (1+\beta_\alpha){\rm Bl}_Z^{*}D_{\alpha}-(c-1)E_{Z}.
		\]
		Hence the following holds:
		\begin{equation}\label{eb}
			-K_{{\rm Bl}_{Z}(X)}\ =\ \sum_{\alpha\in S\setminus S_{P}} (\beta_\alpha+2-c){\rm Bl}_Z^{*}D_{\alpha}+(c-1)(H-E_{Z}).
		\end{equation}
		We have $\beta_\alpha + 2-c\,\ge\, 0$ for all $\alpha\,\in\, S\setminus S_{P}$, so the divisor
		$\sum_{\alpha\in S\setminus S_{P}} (\beta_\alpha+2-c){\rm Bl}_Z^{*}D_{\alpha}$ is effective. Also
		$(c-1)(H-E_{Z})$ is big by Lemma \ref{lemma: big} (recall that $c$ is the codimension of $Z\, \subset\, X$).
		On the other hand, the sum of a big divisor and an effective divisor is big (cf. \cite[Corollary 2.2.7. p. 141]{Lazarsfeld}).
		Therefore, from \eqref{eb} we conclude that the anticanonical line bundle of ${\rm Bl}_{Z}(X)$ is big.
	\end{proof}
	
	\begin{Theorem}\label{Theorem: Fano}
		Let $X$,\, ${\rm Bl}_{Z}X$ and ${\rm Bl}_{Z}$ be as before. Then the anticanonical line bundle of ${\rm Bl}_{Z}(X)$
		is ample (respectively, nef) if and only if $\beta_\alpha-c+2\,>\,0$ (respectively, $\beta_\alpha-c+2\,\geq\,0$) for 
		all $\alpha\,\in\, S\setminus S_P$, where $c$ is the codimension of $Z$ in $X$.
	\end{Theorem}
	
	\begin{proof}
		Consider the isomorphism in \eqref{eb}. From Corollary \ref{cor: ample} (respectively, Theorem \ref{Theorem:blow-up})
		it follows that the anticanonical line bundle of ${\rm Bl}_{Z}(X)$ is ample (respectively, nef) if and only if
		$\beta_\alpha-c+2\,>\, 0$ (respectively, $\beta_\alpha-c+2\,\geq\, 0$) for all $\alpha\,\in\, S\setminus S_P$.
	\end{proof}
	
	\begin{Corollary}\label{cor: Fano}
		Let $X$,\, ${\rm Bl}_{Z}X$ and ${\rm Bl}_{Z}$ be as before. Then ${\rm Bl}_{Z}(X)$ is Fano (respectively, weak-Fano)
		if and only if $\beta_\alpha-c+2\,>\,0$ (respectively, $\beta_\alpha-c+2\,\geq\, 0$) for all $\alpha\,\in\,
		S\setminus S_P$, where $c$ is the codimension of $Z$ in $X$.
	\end{Corollary}
	
	\begin{proof}
		This follows from Theorem \ref{Theorem: Fano} and Theorem \ref{Theorem: Big}.
	\end{proof}
	
	\begin{Corollary}\label{Cor: blow-up point}
		Assume that $P\,=\,B$. Then the anticanonical line bundle of ${\rm Bl}_{Z}X$ is ample (respectively, nef) if and only
		if $3-c\,>\, 0$ (respectively, $3-c\,\ge\, 0$). In particular, ${\rm Bl}_{Z}X$ is Fano (respectively, weak Fano) if and only
		if $3-c\,>\, 0$ (respectively, $3-c\,\ge\, 0$).
	\end{Corollary}
	
	\begin{proof}
		As $P\,=\,B$, we have $S_P\,=\,\emptyset$ and $w_{0,P}\,=\,id$. Therefore, it follows that $\beta_\alpha\,=\,1$ for 
		$\alpha\,\in\, S$. Consequently, Theorem \ref{Theorem: Fano} gives that the anticanonical line bundle of
		${\rm Bl}_{Z}(X)$ is ample (respectively, nef) if and only if $3-c\,>\,0$ (respectively, $3-c\,\ge\, 0$).
	\end{proof}
	
	\begin{Corollary}
		Let $X\,=\,G/B$,\, ${\rm Bl}_{Z}X$ and ${\rm Bl}_{Z}$ be as before. If $Z$ is a smooth Schubert variety in $X$ of
		codimension at most two, then ${\rm Bl}_{Z}X$ is Fano. Moreover, when $Z$ is a smooth Schubert variety in $X$ of
		codimension exactly three, then the anticanonical line bundle of ${\rm Bl}_{Z}X$ is nef but not ample.
	\end{Corollary}
	
	\begin{Corollary}
		Let $P\,=\,B$ (so $X\,=\,G/B$), and let $Z$ be a point of $X$. Then the anticanonical line bundle of
		${\rm Bl}_{Z}(X)$ is ample if and only if $X\,=\, {\rm SL_2(\C)}/B$.
	\end{Corollary}
	
	\begin{proof}
		This follows immediately from Corollary \ref{Cor: blow-up point}.
	\end{proof}
	
	\begin{Corollary}
		Let $Z$ be a point of $X\,=\,G/B$. Then the anticanonical bundle of ${\rm Bl}_{Z}(X)$ is nef if and only
		if $X\,=\, {\rm SL_2}(\mathbb{C})/B$ or $X\,=\, {\rm SL_3}(\mathbb{C})/B$.
	\end{Corollary}
	
	\begin{proof}
		This follows immediately from Corollary \ref{Cor: blow-up point}.
	\end{proof}
	
	\section{Grassmannian and its Generalizations}\label{S4}
	
	Set $G\,=\,{\rm SL}_{n}(\mathbb{C})$ and $P$ to be the maximal parabolic subgroup corresponding to the simple root
	$\alpha_r$. Then $X\,=\,G/P$ is isomorphic to the Grassmannian ${\rm Gr}(r, n)$, parametrizes the $r$--dimensional
	linear subspaces of $\mathbb{C}^n$.	
	
	\begin{Proposition}\label{cor: Grassmannian}
		Take $G\,=\,{\rm SL}_{n}(\mathbb{C})$ and $P$ to be the maximal parabolic subgroup corresponding to the simple root
		$\alpha_r$. Set $Z$ to be a point of $X\,=\, G/P$. Then the anticanonical line bundle of ${\rm Bl}_{Z}(X)$ is ample
		if and only if one of the following three holds:
		\begin{itemize}
			\item [(i)] $(r,\, n)\,=\,(1,\, n)$,
			
			\item [(ii)] $(r,\, n)\,=\,(n-1,\, n)$,
			
			\item [(iii)] $(r,\, n)\,=\,(2,\, 4)$. 
		\end{itemize} 
	\end{Proposition}
	
	\begin{proof}
		Since $P$ is the maximal parabolic subgroup associated to the simple root $\alpha_r$, we have
		$S\setminus S_P\,=\,\{\alpha_{r}\}$. Recall that $\beta_{\alpha_r}\,=\,\langle w_{0, S\setminus
			\{\alpha_{r}\}}(\rho),\, \alpha_{r}^{\vee}\rangle$. Note that 
		$w_{0, S\setminus\{\alpha_{r}\}}(\alpha_{r}^{\vee})\,=\,{\alpha_0}^{\vee}$
		by \cite[Lemma 3.2]{Kannan-Saha}.
		Consequently, we have $\beta_{\alpha_r}\,=\,\langle \rho,\, \alpha_0^{\vee}\rangle\,=\,n-1$. Therefore, by
		Theorem \ref{Theorem: Fano} it follows that the anticanonical line bundle of ${\rm Bl}_{Z}(X)$ is ample if and
		only if $\beta_{\alpha_r}+2\, >\, \dim X$. 
		Observe that $\dim X\,=\,r(n-r)$, so $n+1\, >\,r(n-r)$ if and only if the pair $(r,\,n)$ is one
		of the following forms $(r,\, n)\,=\,(1,\, n)$, $(r,\, n)\,=\,(n-1,\, n)$ and $(r,\, n)\,=\,(2,\, 4)$. 
	\end{proof}
	
	\begin{Corollary}
		Set $G\,=\,{\rm SL}_{n}(\mathbb{C})$ and $P$ to be the maximal parabolic subgroup corresponding to the simple root
		$\alpha_r$. Take $Z$ to be a point of $X\,=\,G/P$. Then the anticanonical line bundle of ${\rm Bl}_{Z}(X)$ is nef
		but not ample if and only if $(r,\, n)\,=\,(2,\, 5)$. 
	\end{Corollary}
	
	\begin{proof}
		{}From the last part of the proof of Proposition \ref{cor: Grassmannian} it follows that
		$n+1\,=\,r(n-r)$ if and only if the pair $(r,\, n)$ is $(2,\,5)$. 
	\end{proof}
	
	\subsection{Cominuscule Grassmannians}
	
	Cominuscule Grassmannians constitute a special type of flag varieties; they generalize the classical 
	Grassmannians.
	
	Let $G$ be a connected, simply-connected, simple affine algebraic group. Note that every $\beta\,\in\, R$
	can be expressed uniquely as $\sum_{i=1}^{\ell}k_i\alpha_i$ such that $k_i\, \in\, {\mathbb Z}$ and
	all of them are either
	non-negative or non-positive. This allows us to define the \emph{height} of a root (relative to $S$) by
	${\rm ht}(\beta)\, =\, \sum_{i=1}^{\ell} k_i$. A simple root $\alpha_{r}\,\in\, S$ is called a \emph{cominuscule}
	simple root if $\alpha_{r}$ occurs with coefficient $1$ in the expression of the highest root
	$\alpha_0$ \cite[p.~119--120]{Billey-Lakshmibai}.
	
	Let $P$ be the maximal parabolic subgroup corresponding to the cominuscule simple root $\alpha_r$. The generalized 
	flag variety $G/P$ associated to the pair $(G, \, P)$ is called a cominuscule Grassmannian. When $G\,=\,{\rm 
		SL}_{n}(\mathbb{C})$ and $P\,\subset\, G$ is a maximal parabolic subgroup, then $G/P$ is a cominuscule Grassmannian 
	\cite{Billey-Lakshmibai}.
	
	\begin{Proposition}\label{prop: cominuscule}
		Let $G$ be a simple algebraic group, $P\,\subset\, G$ a cominuscule parabolic subgroup and $X\,=\,G/P$ the corresponding 
		cominuscule Grassmannian. Let $Z\,\subset\, X$ be a smooth Schubert variety of codimension $c\,\geq\, 2$. Then the 
		anticanonical line bundle of ${\rm Bl}_{Z}(X)$ is ample if and only if $c\,< \,{\rm ht}(\alpha_0)+2$.
In particular, ${\rm Bl}_{Z}(X)$ is Fano if and only if $c\,< \,{\rm ht}(\alpha_0)+2$.
	\end{Proposition}
	
	\begin{proof}
		Since $P$ is the cominuscule maximal parabolic subgroup associate to the cominuscule simple root $\alpha_r$, we 
		have $S\setminus S_P\,=\,\{\alpha_{r}\}$. Recall that $\beta_{\alpha_r}\,=\,\langle w_{0, S\setminus 
			\{\alpha_{r}\}}(\rho),\, {\alpha_{r}}^{\vee}\rangle$. Since $\alpha_{r}$ is a cominuscule simple root, from
		\cite[Lemma 3.2]{Kannan-Saha} it follows that $w_{0, S\setminus\{\alpha_{r}\}}({\alpha_{r}}^{\vee})
		\,=\,{\alpha_0}^{\vee}$. Thus, we have $\beta_{\alpha_r}\,=\,\langle \rho,\, {\alpha_0}^{\vee} \rangle\,
		=\,{\rm ht}(\alpha_0)$. Therefore, from Theorem 
		\ref{Theorem: Fano} it follows that the anticanonical line bundle of ${\rm Bl}_{Z}(X)$ is ample if and only if
we have ${\rm ht}(\alpha_0)+2\, \,>\, c$.
	\end{proof}
	
	\begin{Corollary}\label{cor: cominuscule}
		Take $G$, $P$, $X$ and $Z$ as in Proposition \ref{prop: cominuscule}.
		The anticanonical line bundle of ${\rm Bl}_{Z}(X)$ is nef but not ample if and
		only if $c \,=\, {\rm ht}(\alpha_0)+2$. In particular,  ${\rm Bl}_{Z}(X)$ is weak Fano but not Fano if and
		only if $c \,=\, {\rm ht}(\alpha_0)+2$. 
	\end{Corollary}
	
	\begin{proof}
		{}From the last part of the proof of Proposition \ref{prop: cominuscule} it follows that the anticanonical
		line bundle of ${\rm Bl}_{Z}(X)$ is nef but not ample if and only if we have ${\rm ht}(\alpha_0)+2\,=\,c$. 
	\end{proof}
	
	\begin{Corollary}\label{cor: grassFano}
		Take $G\,=\,{\rm SL}(n,\mathbb{C})$, and set $P\,\subset\, G$ to be the maximal parabolic subgroup
		corresponding to the simple root $\alpha_{r}$. Denote by $X\,=\,G/P$ the Grassmannian ${\rm Gr}(r, n)$, and
		take $Z\,\subset\, X$ to be a smooth Schubert variety of codimension $c\,\geq \,2$. Then the anticanonical
		line bundle of ${\rm Bl}_{Z}(X)$ is ample if and only if $c\,\leq\, n$. In particular, ${\rm Bl}_{Z}(X)$ is Fano if and only if $c\,\leq\, n$.
	\end{Corollary}
	
	\begin{proof}
		This follows from Proposition \ref{prop: cominuscule} and the observation that ${\rm ht}(\alpha_0)\,=\,n-1$ 
		(cf. \cite[p.~66, Table 2]{Humphreys72}).
	\end{proof}
	
	\begin{Corollary}\label{cor: grassweakFano}
		Take $G$, $P$, $X$ and $Z$ as in Corollary \ref{cor: grassFano}.
		The anticanonical line bundle of ${\rm Bl}_{Z}(X)$ is nef but not ample if and only if $c\,=\, n+1$. In particular, ${\rm Bl}_{Z}(X)$ is weak Fano but not Fano if and only if $c\,=\, n+1$.
	\end{Corollary}
	
	\begin{proof}
		This follows from Corollary~\ref{cor: cominuscule} and the observation that ${\rm ht}(\alpha_0)\,=\,n-1$ 
		\cite[p.~66, Table 2]{Humphreys72}.  Proof of the final part follows from Theorem \ref{Theorem: Big} together with the first part.
	\end{proof}
	
	\section{Equivariant vector bundles}\label{S5}
	
	The following 
	proposition records some basic features of the geometry and $T$--equivariant structure on ${\rm 
		Bl}_{Z}X$.
	
	\begin{Proposition}\label{Blow up: Proposition 1.2}
		Let $Z$ be a smooth Schubert variety in $X$ of codimension at least $2$, and let ${\rm Bl}_{Z}\,:\,
		{\rm Bl}_{Z} X\, \longrightarrow\, X$ be the blow-up of $X$ along $Z.$ Then the following three statements hold:
		\begin{enumerate}
			\item [(1)] The action of $B$ on $X$ lifts to ${\rm Bl}_{Z} X.$
			
			\item [(2)] Let $F$ be a $T$-equivariant vector bundle on ${\rm Bl}_{Z} X.$ Then $F$ is nef if and only if the
			restriction $F\big\vert_{\widetilde{C}}$ of $F$ to every $T$--stable curve $\widetilde{C}$ in ${\rm Bl}_{Z} X$ is nef.
			
			\item [(3)] Let $E_{Z}$ denote the exceptional divisor of the blow-up ${\rm Bl}_{Z}$. Let $E$ be a $T$--equivariant 
			vector bundle on $X.$ Then ${\rm Bl}_{Z}^{*}E\otimes \mathcal{O}_{{\rm Bl}_{Z} X}(E_{Z})^{\otimes m}$ is a 
			$T$-equivariant vector bundle for every integer $m.$
		\end{enumerate}
	\end{Proposition}
	
	\begin{proof}
		Since $Z$ is $B$--stable, the group $B$ acts on the normal sheaf $\mathcal{N}_{Z/X}.$ 
		Note that the exceptional divisor of the blow-up ${\rm Bl}_{Z}$ is isomorphic to $\mathbb{P}(\mathcal{N}_{Z/X}^*).$ So, $B$ acts on the exceptional divisor of the blow-up via the linear action of
		$B$ on $\mathcal{N}_{Z/X}^*.$ Since ${\rm Bl}_{Z}$ restricts to an isomorphism over the complement of $Z$, we conclude that the action of $B$ lifts to all of ${\rm Bl}_{Z} X.$ This proves (1).
		
		Since the action of $B$ lifts to ${\rm Bl}_{Z} X$, in particular the action of $T$ lifts,
the proof of (2) goes through exactly as the proof of the analogous statement in \cite[Theorem 3.1]{Biswas-Hanumanthu-Nagaraj}.
		
		Since $T$ acts on the exceptional divisor $E_{Z}$ of ${\rm Bl}_{Z},$ it follows that $\mathcal{O}_{{\rm Bl}_{Z} 
			X}(E_{Z})^{\otimes m}$ is a $T$--equivariant line bundle for every integer $m.$ Hence, if $E$ is a 
		$T$-equivariant vector bundle on $X$, then ${\rm Bl}_{Z}^{*}E \otimes \mathcal{O}_{{\rm Bl}_{Z} X}(E_Z)^{\otimes 
			m}$ is a $T$--equivariant vector bundle on ${\rm Bl}_{Z} X$ for every integer $m.$
	\end{proof}
	
\section*{Acknowledgements}

The first-named author is partially supported by a J. C. Bose Fellowship (JBR/2023/000003). The second-named author gratefully acknowledges the Indian Institute of Technology Delhi for providing a supportive and productive research environment.
	
	\vspace{.1cm}
	
	{\bf Author Contributions.} I.B. and P.S. wrote the manuscript together. Both authors contributed to the project equally.
	\vspace{.1cm}
	
	{\bf Declarations.} 
	\vspace{.1cm}
	
	{\bf Ethical Approval.} Not applicable.
	\vspace{.1cm} 
	
	{\bf Funding.} Not applicable.
	\vspace{.1cm} 
	
	{\bf Availability of data and materials.} No data were used or generated.
	\vspace{.1cm} 
	
	{\bf Conflict of interest.} None of the authors has any conflict of interests.
	\vspace{.1cm}
	
	{\bf Financial interests.} The authors declare they have no financial interests.
	\vspace{.1cm}

	\bibliographystyle{plain}

\end{document}